\documentclass{amsart}

\usepackage{amscd}
\usepackage{mathrsfs}
\usepackage{tipa}
\usepackage{amsfonts}
\usepackage{indentfirst}
\usepackage{amsmath}
\usepackage{amssymb}
\usepackage[all]{xy}
\newtheorem{theo}{Theorem}[section]
\newtheorem{col}[theo]{Corollary}
\newtheorem{lem}[theo]{Lemma}

\newtheorem{prop}[theo]{Proposition}
\newtheorem{rem}[theo]{Remark}
\numberwithin{equation}{section}

\newcommand{\be}{\begin{equation}}
\newcommand{\ee}{\end{equation}}
\newcommand{\bes}{\begin{eqnarray}}
\newcommand{\ees}{\end{eqnarray}}
\newcommand{\bess}{\begin{eqnarray*}}
\newcommand{\eess}{\end{eqnarray*}}

\newcommand{\bali}{\begin{align}}
\newcommand{\eali}{\end{align}}

\begin{document}

\setlength{\baselineskip}{16pt} \pagestyle{myheadings}

\title{REPRESENTATIONs OF FINITE DIMENSIONAL POINTED HOPF ALGEBRAS OVER $\mathbb{Z}_n$}
\author{Ying Zhang}
\address{School of Mathematical Science, Yangzhou University, Yangzhou 225002, China;
and School of Science, Huaihai Institute of Technology, Lianyungang
222005, China} \email{jsgyzy22@sina.com}
\author{Hui-Xiang Chen}
\address{School of Mathematical Science, Yangzhou University, Yangzhou 225002, China}
\email{hxchen@yzu.edu.cn}
\date{}
\maketitle
\noindent
\begin{abstract}
In this paper, we study the representations of the new
finite-dimensional pointed Hopf algebras in positive
characteristic given in \cite{Cib09}. We find that these Hopf
algebras are symmetric algebras. We determine the simple modules and
their projective covers over these Hopf algebras. We show that these
Hopf algebras are of wild representation type.
\end{abstract}

\section{Introduction and Preliminaries}\label{1}

The construction and classification of Hopf algebras play an
important role in the theory of Hopf algebras. During the last few
years several classification results for pointed Hopf algebras were
obtained based on the theory of Nichols algebras \cite{And98,
And02, And08}. In \cite{Cib09}, Cibils, Lauve and Witherspoon
studied Nichols algebras via an embedding in Hopf quiver algebras.
They constructed some new finite dimensional Hopf algebras in
positive characteristic $p$, which are pointed Hopf algebras over
$\mathbb{Z}_n$, the cyclic group of order $n$, where $p|n$. In this
paper, we study these Hopf algebras. We organize the paper as
follows. In this section, we recall some properties of projective
cover and representation theories of Artin algebras, and integrals
in a finite dimensional Hopf algebra, which can be
found in \cite{Aus95, Igl09, Mon93}. In Section \ref{2}, we introduce the
Hopf algebras $\mathcal{B}(V)\# kG$ and its ``lifting"
$H(\lambda,\mu)$ given in \cite{Cib09}, and
investigate some properties of $H(\lambda,\mu)$. We show that
$\mathcal{B}(V)\# kG$ and $H(\lambda,\mu)$ are symmetric algebras. In Section \ref{3}, we
describe the simple modules over $H(\lambda,\mu)$. Then we consider the tensor
products of simple module by using the idea of
\cite{Ch00} and prove that the
tensor product of any two simple modules is indecomposable. Through
computing idempotent elements, we find the projective covers of
these simple modules. In Section \ref{4}, we compute the extensions of
some simple modules over the Hopf algebras and prove that these Hopf
algebras are of wild representation type.

Now we recall some general facts about the representation theory of
a finite dimensional algebra. Let $A$ be a finite dimensional
algebra over an algebraically closed field and
$\widehat{A}=\{S_1,\cdots,S_n\}$ be a complete set of non-isomorphic
simple $A$-modules. Let $P(S)$ denote the projective cover of $S$,
$S\in\widehat{A}$. It is well-known that ${}_A\!A\cong
\bigoplus_{S\in\widehat{A}}P(S)^{{\rm dim} S}$ as left $A$-modules by
Wedderburn-Artin theorem.

Let $H$ be a finite-dimensional Hopf algebra.
A left integral in $H$ is an element $t\in H$ such that
$ht=\varepsilon(h)t$ for all $h\in H$. A right integral in $H$ is
an element $t'\in H$ such that $t'h=\varepsilon(h)t'$ for all $h\in H$.
$\int_H^l$ denotes the space of left integrals, and $\int_H^r$
denotes the space of right integrals. $H$ is called unimodular if
$\int_H^l=\int_H^r$. Note that $\int_H^l$ and $\int_H^r$ are each
one-dimensional (see \cite{Mon93}).

A $k$-algebra $A$ is called symmetric if there exists a nondegenerate
$k$-bilinear form $\beta:A\times A\rightarrow k$, which is
associative and symmetric. A symmetric algebra $A$ is self-injective,
that is, the left regular module $A$ is injective. A finite
dimensional Hopf algebra $H$ is a symmetric algebra if and only if
$H$ is unimodular and $S^2$ is inner, where $S$ is the antipode of $H$
\cite{Lorenz, ObSchn}.

Throughout this paper, we work over an algebraically closed field
$k$ with a positive characteristic $p$. All algebras, Hopf algebras
and modules are finite dimensional over $k$. Unless otherwise
stated, all maps are $k$-linear, dim and $\otimes$ stand for dim$_k$
and $\otimes_k$, respectively.

\section{The Hopf Algebras $\mathcal{B}(V)\# kG$ and $H(\lambda, \mu)$}\label{2}

Let $n>1$ be a positive integer with $p|n$. Let $G=\langle g\rangle$
be the cyclic group of order $n$. Then $kG$ has a 2-dimensional
indecomposable right-right Yetter-Drinfeld module $V$.
$V$ has a basis $\{v_1, v_2\}$ such that the right $kG$-action and $kG$-coaction
are defined by
$$v_1\cdot g=v_1,\ v_2\cdot g=v_1+v_2,\ \rho(v)=v\otimes g,\ v\in V.$$
Then one can form a Nichols algebra $\mathcal{B}(V)$ and the corresponding
pointed Hopf algebra $\mathcal{B}(V)\# kG$.
$\mathcal{B}(V)\# kG$ is a finite dimensional
graded Hopf algebra, which is generated as an algebra by three elements $g$,
$a$ and $b$ (see \cite{Cib09}).

When $p=2$, the generators $g$, $a$ and $b$ of $\mathcal{B}(V)\# kG$ are subject to
the relations:
\begin{eqnarray*}
& g^n=1,\ g^{-1}ag=a,\ g^{-1}bg=a+b,&
\end{eqnarray*}
\begin{eqnarray*}
& a^2=0, \ b^4=0, \ baba=abab, \ {b^2}a=ab^2+aba. &
\end{eqnarray*}
When $p>2$, the generators $g$, $a$ and $b$ of $\mathcal{B}(V)\# kG$ are subject to
the relations:
\begin{eqnarray*}
& g^n=1, \ g^{-1}ag=a,\  g^{-1}bg=a+b, &
\end{eqnarray*}
\begin{eqnarray*}
& a^p=0, \ b^p=0, \ ba=ab+{\frac{1}{2}}a^2. &
\end{eqnarray*}
The coalgebra structure and the antipode of $\mathcal{B}(V)\# kG$ are determined by
\begin{eqnarray*}
& \bigtriangleup(g)=g\otimes g, \ \bigtriangleup(a)={a\otimes
1}+{g\otimes a}, \ \bigtriangleup(b)={b\otimes 1}+{g\otimes b};&
\end{eqnarray*}
\begin{eqnarray*} & \varepsilon(g)=1,\ \varepsilon(a)=\varepsilon(b)=0;&
\end{eqnarray*}
\begin{eqnarray*}
S(g)=g^{-1}, \ S(a)=-g^{-1}a, \ S(b)=-g^{-1}b.
\end{eqnarray*}
Note that $kG$ is the coradical of $\mathcal{B}(V)\# kG$ and $kG$ is a Hopf subalgebra
of $\mathcal{B}(V)\# kG$.

Furthermore, one may construct filtered pointed Hopf algebras as ``lifting"
of $\mathcal{B}(V)\# kG$, that is those whose associated graded algebra is
$\mathcal{B}(V)\# kG$. In the case of $p>2$, Cibils, Lauve and Witherspoon
gave some examples of liftings of $\mathcal{B}(V)\# kG$, which can be described as follows.

Assume $p>2$, and let $\lambda, \mu\in k$. The Hopf algebra $H(\lambda, \mu)$
is generated, as an algebra, by $g$, $a$ and $b$ with the relations
\begin{eqnarray*}
& g^n=1,\  g^{-1}ag=a, \ g^{-1}bg=a+b, &
\end{eqnarray*}
\begin{eqnarray*} & a^p=\lambda(1-g^p), \ b^p=\mu(1-g^p),\
ba=ab+{\frac{1}{2}}a^2.&
\end{eqnarray*}
The coalgebra structure and the antipode of $H(\lambda, \mu)$ are
determined by the same equations as $\mathcal{B}(V)\# kG$. Note that
$kG$ is the coradical of $H(\lambda, \mu)$ and $kG$ is a Hopf
subalgebra of $H(\lambda, \mu)$. Moreover, when $\lambda=\mu=0$,
$H(0,0)=\mathcal{B}(V)\# kG$.

\begin{lem}\label{2.1}
When $p=2$, in $\mathcal{B}(V)\# kG$ we have

 $(1)$ $bg^i=ig^ia+g^ib$, $i\geqslant 0$. In particular, $g^2$ is central
in $\mathcal{B}(V)\# kG$.

$(2)$ $\mathcal{B}(V)\# kG$ is a symmetric Hopf algebra.
\end{lem}
\begin{proof}
(1) It can be proved by induction on $i$ from the relation
$g^{-1}bg=a+b$.

(2) Let $H=\mathcal{B}(V)\# kG$. Then the set $\{g^iabab^3|0\leqslant i\leqslant n-1\}$
are linearly independent in $H$ by \cite[Theorem 3.1 and Corollary 3.4]{Cib09}.
Let $t=(\sum\limits_{0\leqslant i\leqslant n-1}g^i)abab^3$.
Then $t$ is a non-zero element of $H$.
Since $g^n=1$, $g(\sum\limits_{0\leqslant i\leqslant n-1}g^i)=\sum\limits_{0\leqslant i\leqslant n-1}g^i$.
It follows that $gt=t=\varepsilon(g)t$. By the definition of
$H$, we also have $at=(\sum\limits_{0\leqslant i\leqslant n-1}g^i)a^2bab^3=0=\varepsilon(a)t$ and
$bt=(\sum\limits_{0\leqslant i\leqslant n-1}bg^i)abab^3=\sum\limits_{0\leqslant i\leqslant n-1}(ig^ia+g^ib)abab^3
=\sum\limits_{0\leqslant i\leqslant n-1}g^ibabab^3=\sum\limits_{0\leqslant i\leqslant n-1}g^iabab^4=0=\varepsilon(b)t$.
Since $g, a, b$ are generators of $H$, it follows that $\int_H^l=kt$.
On the other hand, we have $a(a+b)=a^2+ab=ab$ and $bg=g(a+b)$. Hence
$tg=(\sum\limits_{0\leqslant i\leqslant n-1}g^i)abab^3g=(\sum\limits_{0\leqslant i\leqslant n-1}g^i)ga(a+b)a(a+b)^3
=(\sum\limits_{0\leqslant i\leqslant n-1}g^i)abab^3=\varepsilon(g)t$.
We also have $ta=(\sum\limits_{0\leqslant i\leqslant n-1}g^i)abab^3a
=(\sum\limits_{0\leqslant i\leqslant n-1}g^i)abab(ab^2+aba)
=(\sum\limits_{0\leqslant i\leqslant n-1}g^i)baba(ab^2+aba)
=0=\varepsilon(a)t$ and
$tb=(\sum\limits_{0\leqslant i\leqslant n-1}g^i)abab^4=0=\varepsilon(b)t$.
Thus, $\int_H^r=kt=\int_H^l$, and so $H$ is unimodular. It is easy to check
that $S^2(g)=g$, $S^2(a)=g^{-1}ag$ and $S^2(b)=g^{-1}bg$. Hence $S^2$ is
inner since $S^2$ is an algebra automorphism. It follows that $H$ is a symmetric
Hopf algebra.
\end{proof}

In the rest of this section, we assume $p>2$. Let $n=p^st$ with
$p\nmid t$ and $s\geqslant 1$. Let $\lambda, \mu\in k$. Now we give some properties of
$H(\lambda,\mu)$.

\begin{lem}\label{2.2} In $H(\lambda,\mu)$, we have

$(1)$ $bg^i=ig^ia+g^ib$, $ba^j=a^jb+\frac{j}{2}a^{j+1}$ and
$bg^ia^j=(i+\frac{j}{2})g^ia^{j+1}+g^ia^jb$ for all $i,\ j\geqslant 0$.
In particular, $g^p$ is central in $H(\lambda,\mu)$.

$(2)$ If $1\leqslant m\leqslant p-1$, then
$$ab^m =\sum_{0\leqslant i\leqslant m}\alpha_{m,i}b^{m-i}a^{i+1},$$
where $\alpha_{m,i}\in k$ with $\alpha_{m,0}=1$,
$\alpha_{m,1}=-\frac{m}{2}$ and $\alpha_{m,2}=\frac{1}{4}m(m-1)$.

$(3)$ If $1\leqslant m\leqslant p-1$, then
$$gb^m =\sum_ {0\leqslant i\leqslant m}\beta_{m,i}b^{m-i}ga^i,$$
where $\beta_{m,i}\in k$ with $\beta_{m,0}=1$,
$\beta_{m,1}=-m$ and $\beta_{m,2}=\frac{3}{4}m(m-1)$.
\end{lem}
\begin{proof}
(1) The first two equalities can be proved by induction on $i$ and $j$, respectively.
The third one follows from the first two equalities.

(2) By the relations of the generators, $ab^m$ can be
expressed as $ab^m = \sum\limits_{0\leqslant i\leqslant m}\alpha_{m,i}b^{m-i}a^{i+1}$
for some $\alpha_{m,i}\in k$. Then for $1\leqslant m< p-1$, by Part (1) we have
\begin{equation*}
\begin{split}
ab^{m+1}&=(\sum\limits_{0\leqslant i\leqslant m}\alpha_{m,i}b^{m-i}a^{i+1})b\\
&=\sum\limits_{0\leqslant i\leqslant m}\alpha_{m,i}b^{m-i}(a^{i+1}b)\\
&=\sum\limits_{0\leqslant i\leqslant m}\alpha_{m,i}b^{m-i}(ba^{i+1}-\frac{i+1}{2}a^{i+2})\\
&=\sum\limits_{0\leqslant i\leqslant m}\alpha_{m,i}b^{m+1-i}a^{i+1}-\sum\limits_{0\leqslant i\leqslant
m}\frac{i+1}{2}\alpha_{m,i}b^{m-i}a^{i+2}.
\end{split}
\end{equation*}
Hence one gets that $\alpha_{m+1,0}=\alpha_{m,0}$, $\alpha_{m+1,m+1}=-\frac{m+1}{2}\alpha_{m,m}$
and $\alpha_{m+1,i}=\alpha_{m,i}-\frac{i}{2}\alpha_{m,i-1}$
for all $1\leqslant i\leqslant m$.
From the definition of $H(\lambda, \mu)$, we know that $\alpha_{1,0}=1$ and $\alpha_{1,1}=-\frac{1}{2}$.
Then by induction on $m$, it is easy to check that
$\alpha_{m,0}=1$, $\alpha_{m,1}=-\frac{m}{2}$ and
$\alpha_{m,2}=\frac{1}{4}m(m-1)$ for all $1\leqslant m\leqslant p-1$.

(3) It is similar to Part (2). We also have $\beta_{1, 0}=1$,
$\beta_{1, 1}=-1$, $\beta_{m+1, 0}=\beta_{m, 0}$,
$\beta_{m+1, m+1}=-\frac{m+2}{2}\beta_{m, m}$ and
$\beta_{m+1, i}=\beta_{m, i}-\frac{i+1}{2}\beta_{m, i-1}$
for all $1\leqslant i\leqslant m<p-1$.
\end{proof}

\begin{lem}\label{2.3} $H(\lambda,\mu)$ is a symmetric Hopf algebra.
\end{lem}
\begin{proof}
From $g^n=1$ and char$k=p$,
it is easy to check that $g(\sum\limits_{0\leqslant i\leqslant n-1}g^i)=\sum\limits_{0\leqslant
i\leqslant n-1}g^i$ and $(1-g^p)(\sum\limits_{0\leqslant i\leqslant
n-1}ig^i)=0$. Since $\{g^ia^{p-1}b^{p-1}|0\leqslant i\leqslant
n-1\}$ are linearly independent (see \cite{Cib09}), $t=(\sum\limits_{0\leqslant
i\leqslant n-1}g^i)a^{p-1}b^{p-1}$ is a non-zero element of
$H(\lambda,\mu)$. Then we have $gt=t=\varepsilon(g)t$,
$at=a^p(\sum\limits_{0\leqslant i\leqslant
n-1}g^i)b^{p-1}=\lambda(1-g^p)(\sum\limits_{0\leqslant i\leqslant
n-1}g^i)b^{p-1}=0=\varepsilon(a)t$ and
\begin{equation*}
\begin{split}
bt&=(\sum_{0\leqslant i\leqslant
n-1}bg^ia^{p-1})b^{p-1}\\
&=[\sum_{0\leqslant i\leqslant
n-1}(i+\frac{p-1}{2})g^ia^p+g^ia^{p-1}b]b^{p-1}\\
&=a^p(\sum_{0\leqslant i\leqslant
n-1}ig^i)b^{p-1}+\frac{p-1}{2}a^p(\sum_{0\leqslant i\leqslant
n-1}g^i)b^{p-1}+b^p(\sum_{0\leqslant i\leqslant n-1}g^i)a^{p-1}\\
 &=0=\varepsilon(b)t.
\end{split}
\end{equation*}
Since $g,a,b$ are generators of $H(\lambda,\mu)$, one gets that
$\int_H^l=kt$. On the other hand, since $ba=a(b+\frac{1}{2}a)$, we
have $(a+b)^{p-1}=b^{p-1}+a\sum\limits_{0\leqslant j\leqslant
p-2}\alpha_ja^jb^{p-2-j}$ for some $\alpha_j\in k$. Hence
\begin{equation*}
\begin{split}
tg&=(\sum_{0\leqslant i\leqslant
n-1}g^i)a^{p-1}b^{p-1}g\\
&=(\sum_{0\leqslant i\leqslant
n-1}g^i)a^{p-1}g(a+b)^{p-1}\ \ ({\rm by}\ bg=g(a+b))\\
&=(\sum_{0\leqslant i\leqslant
n-1}g^i)a^{p-1}[b^{p-1}+a\sum_{0\leqslant j\leqslant p-2}\alpha_ja^jb^{p-2-j}]\\
&=t+a^p(\sum_{0\leqslant i\leqslant
n-1}g^i)(\sum_{0\leqslant j\leqslant p-2}\alpha_ja^jb^{p-2-j})\\
&=t=\varepsilon(g)t,\\
\end{split}
\end{equation*}
\begin{equation*}
\begin{split}
ta&=(\sum_{0\leqslant i\leqslant
n-1}g^i)a^{p-1}b^{p-1}a\\
&=(\sum_{0\leqslant i\leqslant
n-1}g^i)a^{p-1}a(b+\frac{1}{2}a)^{p-1}\ ({\rm by}\  ba=a(b+\frac{1}{2}a))\\
&=a^p(\sum_{0\leqslant i\leqslant n-1}g^i)(b+\frac{1}{2}a)^{p-1}\\
&=0=\varepsilon(a)t\\
\end{split}
\end{equation*}
and
\begin{equation*}
\begin{split}
tb&=b^p(\sum_{0\leqslant i\leqslant
n-1}g^i)a^{p-1}=0=\varepsilon(b)t,
\end{split}
\end{equation*}
where we use the facts that $a^p=\lambda(1-g^p)$ and $b^p=\mu(1-g^p)$
are central elements in $H(\lambda, \mu)$.
Thus, $\int_H^r=kt=\int_H^l$, and so $H$ is unimodular. It is easy to check
that $S^2$ is inner. It follows that $H$ is a symmetric Hopf algebra.
\end{proof}

\begin{lem}\label{2.4}
Let $J$ be the Jacobson radical of $H(\lambda,\mu)$. Then

$(1)$ If $t=1$, then $a, b\in J$.

$(2)$ If $t>1$ and $\lambda\mu\neq0$, then $a,b \not\in J$.
\end{lem}
\begin{proof}
(1) Assume $t=1$. Then $n=p^s$. Since char$k$=$p$ and $a^p=\lambda(1-g^p)$, we have
$a^n=a^{p^s}=[\lambda(1-g^p)]^{p^{s-1}}=\lambda ^{p^{s-1}}(1-g^{p^s})=
\lambda^{p^{s-1}}(1-g^n)=0$.
On the other hand, we have $ag=ga$ and $ab=ba-\frac{1}{2}a^2=(b-\frac{1}{2}a)a$.
Hence $aH(\lambda, \mu)=H(\lambda,\mu)a$, and consequently $H(\lambda, \mu)a$ is
equal to the ideal $\langle a\rangle$ of $H(\lambda, \mu)$ generated by $a$.
It follows that $(H(\lambda, \mu)a)^n=H(\lambda, \mu)a^n=0$.
Thus, $H(\lambda, \mu)a\subseteq J$, and so $a\in J$. Similarly, we have $b^n=0$.
Consider the quotient algebra $H(\lambda, \mu)/\langle a\rangle$ of $H(\lambda, \mu)$
modulo $\langle a\rangle$. Then $H(\lambda, \mu)/\langle a\rangle$ is generated,
as an algebra, by $\overline{g}$ and $\overline{b}$. In this case, we have
$\overline{g}\overline{b}=\overline{b}\overline{g}$.
It follows that the ideal $\langle\overline{b}\rangle$ of $H(\lambda, \mu)/\langle a\rangle$
generated by $\overline{b}$ satisfies $\langle\overline{b}\rangle^n=0$.
Therefore, $b\in J$.

(2) Assume $t>1$ and $\lambda\mu\neq0$. Then $g^{p^m}\neq1$ for all $m\geqslant 0$. Hence
$a^{p^m}=\lambda^{p^{m-1}}(1-g^{p^m})\neq0$ for all $m\geqslant 0$. This means
that $a$ is not a nilpotent element, and so $a\notin J$. Similarly,
$b\notin J$.
\end{proof}

\begin{lem}\label{2.5} If $\lambda\neq0$, then $H(\lambda,\mu)\cong
H(1,\lambda^{-1}\mu)$.
\end{lem}
\begin{proof}
Assume $\lambda\neq0$.
Let $g$, $a$, $b$ and $g_0$, $a_0$, $b_0$ denote the generators of $H(\lambda,\mu)$ and
$H(1,\lambda^{-1}\mu)$, respectively. Then in $H(1,\lambda^{-1}\mu)$ we have
$g_0^n=1$, $g_0^{-1}(\lambda^{\frac{1}{p}}a_0)g=\lambda^{\frac{1}{p}}a_0$,
$g_0^{-1}(\lambda^{\frac{1}{p}}b_0)g_0=\lambda^{\frac{1}{p}}a_0+\lambda^{\frac{1}{p}}b_0$.
$(\lambda^{\frac{1}{p}}a_0)^p=\lambda(1-g_0^p)$, $(\lambda^{\frac{1}{p}}b_0)^p=\mu(1-g_0^p)$,
and $(\lambda^{\frac{1}{p}}b_0)(\lambda^{\frac{1}{p}}a_0)
=(\lambda^{\frac{1}{p}}a_0)(\lambda^{\frac{1}{p}}b_0)+{\frac{1}{2}}(\lambda^{\frac{1}{p}}a_0)^2$.
It follows that there is an algebra map
$\varphi:H(\lambda,\mu)\rightarrow H(1,\lambda^{-1}\mu)$
such that $\varphi(g)=g_0$, $\varphi(a)=\lambda^{\frac{1}{p}}a_0$ and
$\varphi(b)=\lambda^{\frac{1}{p}}b_0$. It is easy to see that
$\varphi$ is a Hopf algebra homomorphism. Similarly,
there exists a Hopf algebra homomorphism
$\psi: H(1, \lambda^{-1}\mu)\rightarrow H(\lambda, \mu)$
such that $\psi(g_0)=g$, $\psi(a_0)=\lambda^{-\frac{1}{p}}a$ and
$\psi(b_0)=\lambda^{-\frac{1}{p}}b$. Obviously, $\varphi\circ\psi={\rm id}$
and $\psi\circ\varphi={\rm id}$, and so $H(\lambda,\mu)\cong
H(1,\lambda^{-1}\mu)$.
\end{proof}

\section{Simple modules and Projective Modules over $H(\lambda, \mu)$}\label{3}

Throughout this section, assume $p>2$. Let $n=p^st$ with $p\nmid t$ and $s\geqslant 1$.
Let $\xi$ be a $t$-$th$ primitive root of unity in $k$. Let $\lambda, \mu\in k$.
We will investigate simple modules and projective modules
over $H(\lambda, \mu)$ in this section. Note that $kG$ is the coradical of $H(\lambda, \mu)$.

Since $p|n$, we know that $kG$ is not semisimple. It has $t$ non-isomorphic simple modules, which are
all 1-dimensional and given by the corresponding algebra homomorphisms
$\rho_i: kG\rightarrow k$, $\rho_i(g)=\xi^i$, $0\leqslant i\leqslant t-1$.
Moreover, $kG$ has $n$ non-isomorphic indecomposable modules, which can be
described by the matrix representations as follows:
\[
\rho_{r,i}(g)=\left(\begin{array}{ccccc}
\xi^i&1&\cdots&0&0\\
0&\xi^i&\cdots&0&0\\
\cdots&\cdots&\cdots&\cdots&\cdots\\
0&0&\cdots&\xi^i&1\\
0&0&\cdots&0&\xi^i
\end{array}\right)_{r\times r}
\]
where $1\leqslant r\leqslant p^s$ and $0\leqslant i\leqslant t-1$ (see \cite{Don09}).

\begin{theo}\label{3.1}If $t=1$, there is only one simple module $S$
over $H(\lambda,\mu)$, which is $1$-dimensional and given by $g\cdot v=v$, $a\cdot v=0$
and $b\cdot v=0$ for all $v\in S$.
In particular, $H(\lambda,\mu)$ is a local algebra in this case.
\end{theo}

\begin{proof}
Assume $t=1$. Then by Lemma \ref{2.4}(1), we know that $a,b\in J$,
the Jacobson radical of $H(\lambda, \mu)$, and
$H(\lambda,\mu)/\langle a, b\rangle\cong kG$, where $\langle a,
b\rangle$ is the ideal of $H(\lambda, \mu)$ generated by $a$ and
$b$. Hence the theorem follows.
\end{proof}

In the rest of this section, assume $t>1$.

\begin{lem}\label{3.2}
Let $M$ be an $H(\lambda,\mu)$-module. If there exists an element
$0\neq v\in M$ such that $g\cdot v=\alpha v$ and $a\cdot v=\beta v$
for some $\alpha,\beta\in k$ with $\beta\neq 0$, then the following statements
holds:

$(1)$ If $1\leqslant m\leqslant p-1$, then
$$ab^m \cdot v= \sum_{0\leqslant j\leqslant m}
\alpha_{m,j}b^j\cdot v\ \mbox{ and }\ gb^m\cdot v=\sum_ {0\leqslant j\leqslant m}\beta_{m, j}b^j\cdot v,$$
where $\alpha_{m,j}, \beta_{m, j}\in k$
with $\alpha_{m,m}=\beta$, $\alpha_{m, m-1}=-{\frac{m}{2}}\beta^2$,
$\beta_{m, m}=\alpha$, and $\beta_{m, m-1}=-m\alpha\beta$.

$(2)$ $N ={\rm span}\{v, b\cdot v, \cdots, b^{p-1}\cdot v\}$ is an
submodule of $M$.

$(3)$ $\{v, b\cdot v, \cdots, b^{p-1}\cdot v\}$ are linearly
independent.

$(4)$ Consider the actions of $g$ and $a$ on $N$.
Then $\alpha$ and $\beta$ are the only eigenvalues
of $g$ and $a$, respectively, with multiplicity $p$.
Moreover, $v$ is the unique common eigenvector of $g$ and $a$
up to a non-zero scale multiple.

$(5)$ $N$ is a simple $H(\lambda,\mu)$-module.
\end{lem}
\begin{proof}
(1) It follows from Parts (2) and (3) of Lemma \ref{2.2}.

(2) Since $b^p=\mu(1-g^p)$, it follows from Part (1).

(3) Suppose that $\{v, b\cdot v, \cdots, b^{p-1}\cdot v\}$
are linearly dependent. Since $v\neq 0$, there exists an $m$ with $0\leqslant m< p-1$
such that $\{v, b\cdot v, \cdots, b^m\cdot v\}$ are linearly independent,
but $\{v, b\cdot v, \cdots, b^m\cdot v, b^{m+1}\cdot v\}$ are linearly
dependent. Hence there are some $\alpha_i\in k$ such that
$b^{m+1}\cdot v=\sum\limits_{0\leqslant i\leqslant m}\alpha_ib^i\cdot v$.
Thus, $ab^{m+1}\cdot v=\sum\limits_{0\leqslant i\leqslant m}\alpha_iab^i\cdot v$.
By Part (1), we have $ab^{m+1} \cdot v= \sum\limits_{0\leqslant j\leqslant m+1}
\alpha_{m+1,j}b^j\cdot v=\beta b^{m+1}\cdot v+ \sum\limits_{0\leqslant j\leqslant m}
\alpha_{m+1,j}b^j\cdot v$ and
$$\begin{array}{rcl}
\sum\limits_{0\leqslant i\leqslant m}\alpha_iab^i\cdot v
&=&\sum\limits_{0\leqslant i\leqslant m}
\sum\limits_{0\leqslant j\leqslant i}\alpha_i\alpha_{i, j}b^j\cdot v\\
&=&\alpha_m\beta b^m\cdot v+\sum\limits_{0\leqslant j\leqslant m-1}\gamma_jb^j\cdot v,\\
\end{array}$$
where $\gamma_j\in k$ for $0\leqslant j\leqslant m-1$.
Hence we have
$$\begin{array}{rcl}
ab^{m+1}\cdot v-\beta b^{m+1}\cdot v&=&\sum\limits_{0\leqslant j\leqslant m}
\alpha_{m+1,j}b^j\cdot v\\
&=&-\frac{m+1}{2}\beta^2b^m\cdot v+\sum\limits_{0\leqslant j\leqslant m-1}
\alpha_{m+1,j}b^j\cdot v\\
\end{array}$$
and
$$a(\sum\limits_{0\leqslant i\leqslant m}\alpha_ib^i\cdot v)-
\beta(\sum\limits_{0\leqslant i\leqslant m}\alpha_ib^i\cdot v)
=\sum\limits_{0\leqslant j\leqslant m-1}(\gamma_j-\alpha_j\beta)b^j\cdot v.$$
It follows that $-\frac{m+1}{2}\beta^2b^m\cdot v+\sum\limits_{0\leqslant j\leqslant m-1}
\alpha_{m+1,j}b^j\cdot v=\sum\limits_{0\leqslant j\leqslant m-1}(\gamma_j-\alpha_j\beta)b^j\cdot v.$
Since $-\frac{m+1}{2}\beta^2\neq0$, one gets that $\{v, b\cdot v, \cdots, b^m\cdot v\}$
are linearly dependent, a contradiction.

(4) It follows from Parts (1) and (3).

(5) Let $N_0$ be a non-zero submodule of $N$. Then $N_0$ must
contain a common eigenvector of $g$ and $a$. Hence $v\in N_0$ by
Part (4), and so $N_0=N$. This shows that $N$ is a simple module.
\end{proof}

Now we will compute simple modules over $H(\lambda,\mu)$. Note that
$H(\lambda, \mu)=\mathcal{B}(V)\# kG$ if $\lambda=\mu=0$.
We first consider the case of $\lambda=0$.

\begin{theo}\label{3.3} Let $\mu\in k$. Then there are $t$ non-isomorphic simple modules $T_i$ over $H(0,\mu)$,
$0\leqslant i\leqslant t-1$. Each $T_i$ is $1$-dimensional and given by
$$g\cdot v=\xi^iv,\  a\cdot v=0,\ b\cdot v=\mu^\frac{1}{p}(1-\xi^i)v,\ v\in T_i.$$
\end{theo}
\begin{proof}
Let $0\leqslant i\leqslant t-1$. Then it is easy to see that
there is an algebra map $\rho_i: H(0, \mu)\rightarrow k$ such that
$\rho_i(g)=\xi^i$, $\rho_i(a)=0$ and $\rho_i(b)=\mu^\frac{1}{p}(1-\xi^i)$.
It follows that $T_0, T_1, \cdots, T_{t-1}$ given in the theorem are non-isomorphic
1-dimensional simple $H(0, \mu)$-modules.

By the proof of Lemma \ref{2.4}(1), one knows that the ideal
$\langle a\rangle$ of $H(0, \mu)$ generated by $a$ is equal to $H(0,
\mu)a=aH(0, \mu)$. Since $a^p=0$, $\langle a\rangle^p=(H(0,
\mu)a)^p=H(0, \mu)a^p=0$. Hence $\langle a\rangle\subseteq J$, the
Jacobson radical of $H(0, \mu)$. Thus, any simple $H(0, \mu)$-module
is a simple module over the quotient algebra $H(0, \mu)/\langle
a\rangle$. However, $H(0, \mu)/\langle a\rangle$ is a commutative
algebra and $k$ is an algebraically closed field. It follows that
any simple $H(0, \mu)$-module is 1-dimensional and determined by an
algebra map from $H(0, \mu)$ to $k$. Now let $\rho: H(0,
\mu)\rightarrow k$ be an algebra map. Then $\rho(a)=0$. Since
$\rho(g)^n=\rho(g^n)=\rho(1)=1$, $\rho(g)=\xi^i$ for some
$0\leqslant i\leqslant t-1$. Since $b^p=\mu(1-g^p)$,
$\rho(b)^p=\mu(1-\rho(g)^p)=\mu(1-\xi^{ip})=(\mu^{\frac{1}{p}}(1-\xi^i))^p$,
and so $\rho(b)=\mu^{\frac{1}{p}}(1-\xi^i)$. Thus, $\rho=\rho_i$.
This completes the proof.
\end{proof}

For the case of $\lambda\neq 0$, by Lemma \ref{2.5}, we may assume
$\lambda=1$. Let $S_0$ be the trivial $H(1, \mu)$-module given by
the counit $\varepsilon: H(1, \mu)\rightarrow k$. Then dim$S_0=1$,
and
$$g\cdot v=v,\ a\cdot v=0,\ b\cdot v=0,\ v\in S_0.$$

Now let $A$ be the subalgebra of $H(1, \mu)$ generated by $g$ and $a$.
Then $A$ is a Hopf subalgebra of $H(1, \mu)$. Hence $H(1, \mu)$ is a free
right (left) $A$-module \cite{Mon93}. Note that $A$ is a commutative algebra.
For $1\leqslant i\leqslant t-1$, there is an algebra map
$\rho_i: A\rightarrow k$ defined by $\rho_i(g)=\xi^i$ and $\rho_i(a)=1-\xi^i$.
Let $X_i$ denote the corresponding left $A$-module. Then dim$X_i=1$,
$g\cdot x=\xi^ix$ and $a\cdot x=(1-\xi^i)x$ for all $x\in X_i$.
Let $S_i=H(1, \mu)\otimes_AX_i$. Then $S_i$ is a non-zero left cyclic
$H(1, \mu)$-module generated by $1\otimes x$, where $0\neq x\in X_i$.

\begin{theo}\label{3.4} Let $0\leqslant i\leqslant t-1$. Then we have

$(1)$ $S_0, S_1, \cdots, S_{t-1}$ are non-isomorphic simple $H(1, \mu)$-modules.

$(2)$ If $i\neq 0$, ${\rm dim}S_i=p$ and there is a $0\neq v\in S_i$ such that
$g\cdot v=\xi^i v$ and
$a\cdot v=(1-\xi^i)v$. Moreover,
$\{v, b\cdot v, \cdots, b^{p-1}\cdot v\}$ is a basis of $S_i$.

$(3)$ If $M$ is a simple $H(1, \mu)$-module, then $M$ is isomorphic to some $S_i$.
\end{theo}
\begin{proof}
We have already known that $S_0$ is a simple $H(1, \mu)$-module
and dim$S_0=1$. Now let $1\leqslant i\leqslant t-1$ and take $0\neq x\in X_i$.
Let $v=1\otimes x\in S_i$. Then $g\cdot v=\xi^i v$ and
$a\cdot v=(1-\xi^i)v$. Since $S_i$ is a cyclic $H(1, \mu)$-module generated by $v$,
it follows from Lemma \ref{3.2} that $S_i$ is a simple $H(1, \mu)$-module
with dim$S_i=p$. Moreover, $\{v, b\cdot v, \cdots, b^{p-1}\cdot v\}$ is a basis of $S_i$,
and $v$ is the unique common eigenvector of the actions of $g$ and $a$ on $S_i$
up to a non-zero scale multiple. Thus, $S_0, S_1, \cdots, S_{t-1}$ are non-isomorphic
simple $H(1, \mu)$-modules. This shows Parts (1) and (2).

Now let $M$ be a simple $H(1,\mu)$-module. Since $k$ is an
algebraically closed field and $ga=ag$, there is a non-zero vector
$v\in M$ such that $g\cdot v=\alpha v$ and $a\cdot v=\beta v$ for
some $\alpha, \beta\in k$. Hence $A\cdot v=kv$. Since $g^n=1$,
$\alpha^n=\alpha^{p^st}=(\alpha^t)^{p^s}=1$. Hence $\alpha^t=1$, and
consequently $\alpha=\xi^i$ for some $0\leqslant
i\leqslant t-1$. Since $a^p=1-g^p$, we have
$\beta^p=1-\xi^{ip}=(1-\xi^i)^p$. It follows that $\beta=1-\xi^i$.
Since $M$ is a simple $H(1, \mu)$-module and $H(1,
\mu)=\sum\limits_{0\leqslant j\leqslant p-1}b^jA$, one gets that
$M=H(1, \mu)\cdot v={\rm span}\{v, b\cdot v, \cdots, b^{p-1}\cdot
v\}$. We divide the discussion into the following two cases.

For the case: $i=0$. In this case, $g\cdot v=v$, $a\cdot v=0$ and
$b^p\cdot v=\mu(1-g^p)\cdot v=0$. Hence there is an integer $m$ with
$0\leqslant m\leqslant p-1$ such that $b^m\cdot v\neq0$ but
$b^{m+1}\cdot v=0$. If $m=0$, then $g\cdot v=v$, $a\cdot v=0$ and
$b\cdot v=0$. Hence $M=kv\cong S_0$ since $M$ is simple. If $m>0$,
then by Lemma \ref{2.2} it follows that $ab^m\cdot v=0$ and
$gb^m\cdot v=b^m\cdot v$. Thus, $k\{b^m\cdot v\}$ is a non-zero
$H(1,\mu)$-submodule of $M$, and so $M=k(b^m\cdot v)\cong S_0$ since
$M$ is simple. In this case, $v=\gamma b^m\cdot v$ for some $0\neq
\gamma\in k$, which implies that $b\cdot v=0$, and so $m=0$, a
contradiction.

For the case: $1\leqslant i\leqslant t-1$. In this case, $a\cdot
v=(1-\xi^i)v\neq0$. Since $M$ is a simple $H(1, \mu)$-module, it
follows from Lemma \ref{3.2} that $k\{v, b\cdot v, \cdots,
{b^{p-1}\cdot v}\}$ is a basis of $M$. In this case, $M$ is
isomorphic to $S_i$. In fact, let $0\neq x\in X_i$. Then there is an
$A$-module isomorphism $f: X_i\rightarrow kv$, $f(x)=v$, where $kv$
is obviously an $A$-submodule of $M$. Since $M=H(1, \mu)\cdot v$, we
have an $H(1, \mu)$-module epimorphism
$$\psi: S_i=H(1, \mu)\otimes_AX_i\xrightarrow{{\rm id}\otimes f}
H(1, \mu)\otimes_A(kv)\xrightarrow{\cdot}M$$
given by $\psi(h\otimes x)=h\cdot f(x)=h\cdot v$, $h\in H(1, \mu)$.
Since both $S_i$ and $M$ are simple, $\psi$ must be an isomorphism.
\end{proof}

For any integer $i$, let $0\leqslant\overline{i}\leqslant t-1$
with $\overline{i}\equiv i$ (mod $t)$. For any positive integer $m$,
let $I_m$ denote the identity $m\times m$-matrix over $k$. For any matrix $X$ over
$k$, let $r(X)$ denote the rank of $X$.

For $1\leqslant i, j\leqslant t-1$, let $\{b^{i_1}\cdot v\}_{0\leqslant i_1\leqslant
p-1}$ and $\{b^{j_1}\cdot w\}_{0\leqslant j_1\leqslant p-1}$ be the basis of
$S_i$ and $S_j$ as stated in Theorem \ref{3.4}, respectively. Then $\{b^{i_1}\cdot v\otimes
b^{j_1}\cdot w\}_{0\leqslant i_1,j_1\leqslant p-1}$ is a basis of
$S_i\otimes S_j$. For any $0\neq u=\sum
x_{i_1,j_1}b^{i_1}\cdot v \otimes b^{j_1}\cdot w\in S_i\otimes S_j$, let
$h(u)={\rm max}\{i_1+j_1|x_{i_1,j_1}\neq0\}$ and let
$$u(1)={\rm max}\{i_1|x_{i_1,j_1}\neq0 \mbox{ for some $j_1$}\}\mbox{ and }
u(2)={\rm max}\{j_1|x_{u(1),j_1}\neq0\}.$$
With the above notations, we have the following lemma.

\begin{lem}\label{3.5}
Let $0\neq u\in S_i\otimes S_j$ with $h(u)=u(1)=l>0$.
Assume $v_1=g\cdot u-\xi^{i+j}u\neq 0$. Then

$(1)$ $h(v_1)< l$.

$(2)$ If $v_1(2)=0$, then there is an element $u'\in S_i\otimes S_j$
with $h(u')\leqslant l$ and $u'(1)=$
$v_1(1)$ such that $g\cdot u''-\xi^{i+j}u''=0$,
or $(g\cdot u''-\xi^{i+j}u'')(1)<v_1(1)$, where $u''=u+u'$.

$(3)$ If $v_1(2)>0$, then there is an element $u'\in S_i\otimes S_j$
with $h(u')\leqslant l$ and $u'(1)=$
$v_1(1)$ such that $g\cdot u''-\xi^{i+j}u''=0$,
or $(g\cdot u''-\xi^{i+j}u'')(1)<v_1(1)$,
or $(g\cdot u''-$ $\xi^{i+j}u'')(1)=v_1(1)$ and
$(g\cdot u''-\xi^{i+j}u'')(2)<v_1(2)$, where $u''=u+u'$.
\end{lem}

\begin{proof}
Let $v_1(1)=m$ and $v_1(2)=s$.

(1) It follows from Lemma \ref{3.2}(1).

(2) Assume $s=0$. By Part (1), we have $0\leqslant m<l$. Hence
$v_1=\alpha b^m\cdot v\otimes w+\sum\limits_{i_1<m}\alpha_{i_1, j_1}b^{i_1}\cdot v\otimes b^{j_1}\cdot w$
for some $\alpha$, $\alpha_{i_1, j_1}\in k$ with $\alpha\neq 0$.
Take $u'=\alpha\xi^{-(i+j)}(1-\xi^j)^{-1}b^m\cdot v\otimes b\cdot w$ and let $u''=u+u'$.
Then $h(u')=m+1\leqslant l$, $u'(1)=m$ and
$$g\cdot u'-\xi^{i+j}u'=-\alpha b^m\cdot v\otimes w+\sum\limits_{i_1<m, j_1\leqslant 1}
\beta_{i_1, j_1}b^{i_1}\cdot v\otimes b^{j_1}\cdot w.$$
Since $g\cdot u''-\xi^{i+j}u''=v_1+g\cdot u'-\xi^{i+j}u'$, we know that
$g\cdot u''-\xi^{i+j}u''=0$, or $(g\cdot u''-\xi^{i+j}u'')(1)<m$.

(3) Assume $s>0$. Then
$$v_1=\sum\limits_{0\leqslant j_1\leqslant s}\alpha_{j_1}b^m\cdot v
\otimes b^{j_1}\cdot w+\sum\limits_{i_1<m}\alpha_{i_1, j_1}b^{i_1}\cdot v\otimes b^{j_1}\cdot w$$
for some $\alpha_{j_1}$, $\alpha_{i_1, j_1}\in k$ with $\alpha_s\neq 0$.
Note that $m+s\leqslant h(v_1)<l\leqslant p-1$. Hence $s<p-1$ and so $1<s+1<p$.
Let $u'=\alpha_s(s+1)^{-1}\xi^{-(i+j)}(1-\xi^j)^{-1}b^m\cdot v\otimes b^{s+1}\cdot w$
and $u''=u+u'$. Then $h(u')=m+s+1\leqslant l$, $u'(1)=m$ and
$$g\cdot u'-\xi^{i+j}u'=-\alpha_s b^m\cdot v\otimes b^s\cdot w+\sum\limits_{j_1<s}
\beta_{j_1}b^m\cdot v\otimes b^{j_1}\cdot w+\sum\limits_{i_1<m, j_1\leqslant s+1}
\beta_{i_1, j_1}b^{i_1}\cdot v\otimes b^{j_1}\cdot w.$$
Since $g\cdot u''-\xi^{i+j}u''=v_1+g\cdot u'-\xi^{i+j}u'$, we know that
$g\cdot u''-\xi^{i+j}u''=0$, or $(g\cdot u''-\xi^{i+j}u'')(1)<m$,
or $(g\cdot u''-\xi^{i+j}u'')(1)=m$ and
$(g\cdot u''-\xi^{i+j}u'')(2)<s$.
\end{proof}

\begin{theo}\label{3.6}
Let $0\neq u\in S_i\otimes S_j$ with $h(u)=u(1)=l>0$. If $g\cdot u\neq\xi^{i+j}u$,
then there is an element $\overline{u}\in S_i\otimes S_j$ with
$h(\overline{u})\leqslant l$ and $\overline{u}(1)<l$
such that $g\cdot\underline{u}=\xi^{i+j}\underline{u}$, where $\underline{u}=u+\overline{u}$.
\end{theo}

\begin{proof}
Let $u_1=u$, $v_1=g\cdot u_1-\xi^{i+j}u_1\neq 0$, $m_1=v_1(1)$ and $s_1=v_1(2)$.
Then it follows from Lemma \ref{3.5} that $m_1<l$ and
there is an elements $u_1'\in S_i\otimes S_j$
with $h(u_1')\leqslant l$ and $u_1'(1)=m_1<l$ such that $g\cdot u_2=\xi^{i+j}u_2$,
or $(g\cdot u_2-\xi^{i+j}u_2)(1)<m_1$,
or $(g\cdot u_2-\xi^{i+j}u_2)(1)=m_1$ and $(g\cdot u_2-\xi^{i+j}u_2)(2)<s_1$, where $u_2=u_1+u_1'$.
If $g\cdot u_2=\xi^{i+j}u_2$, then the theorem follows. Otherwise,
let $v_2=g\cdot u_2-\xi^{i+j}u_2\neq0$, $v_2(1)=m_2$ and $v_2(2)=s_2$. Since $u_1(1)=l$ and $u_1'(1)=m_1<l$,
$u_2(1)=l$, and so $h(u_2)=l$. By replacing $u_1$ with $u_2$, it follows from Lemma \ref{3.5}
that there is an $u_2'\in S_i\otimes S_j$ with $h(u_2')\leqslant l$ and $u_2'(1)=m_2<l$
such that $g\cdot u_3=\xi^{i+j}u_3$, or $(g\cdot u_3-\xi^{i+j}u_3)(1)<m_2$,
or $(g\cdot u_3-\xi^{i+j}u_3)(1)=m_2$ and $(g\cdot u_3-\xi^{i+j}u_3)(2)<s_2$, where $u_3=u_2+u_2'$.
Since $h(u_1')\leqslant l$ and $h(u_2')\leqslant l$, $h(u_1'+u_2')\leqslant l$. Furthermore, we have
$u_2'(1)=m_2<m_1=u_1'(1)$, or $u_2'(1)=m_2=m_1=u_1'(1)$ and $u_2'(2)=s_2<s_1$.
It follows that $(u_1'+u_2')(1)\leqslant m_1<l$. We also have $u_3=u_2+u_2'=u_1+u_1'+u_2'$.
If $g\cdot u_3=\xi^{i+j}u_3$, then the theorem follows. Otherwise,
let $v_3=g\cdot u_3-\xi^{i+j}u_3\neq0$, $v_3(1)=m_3$ and $v_3(2)=s_2$. Since $u_2(1)=l$ and $u_2'(1)=m_2<l$,
$u_3(1)=l$, and so $h(u_3)=l$. Then we may repeat the above procedure by replacing $u_2$ with $u_3$, and continue.
Thus one may get a series of elements $u_1', u_2', u_3', \cdots$ in $S_i\otimes S_j$
with $h(u_q')\leqslant l$ and $u_q'(1)=m_q<l$ such that $g\cdot u_{q+1}=\xi^{i+j}u_{q+1}$,
or $m_{q+1}:=(g\cdot u_{q+1}-\xi^{i+j}u_{q+1})(1)<m_q$,
or $m_{q+1}:=(g\cdot u_{q+1}-\xi^{i+j}u_{q+1})(1)=m_q$ and $s_{q+1}:=(g\cdot u_{q+1}-\xi^{i+j}u_1)(2)<s_q$,
where $u_{q+1}=u_q+u_q'$, $q=1, 2, 3, \cdots$.

We claim that the above procedure will stop.
In fact, if $v_q=g\cdot u_q-\xi^{i+j}u_q\neq 0$ for all $q\geqslant 1$,
then $m_{q+1}<m_q$, or $m_{q+1}=m_q$ and $s_{q+1}<s_q$
for all $q\geqslant 1$. Since $l>m_1\geqslant m_2\geqslant m_3\geqslant\cdots\geqslant 0$,
there is a $q\geqslant 1$ such that $m_q=m_{q+1}=m_{q+2}=\cdots$.
Then it follows that $s_q>s_{q+1}>s_{q+2}>\cdots\geqslant 0$.
This is impossible. Thus, there exists an integer $m\geqslant 1$ such that
$v_q=g\cdot u_q-\xi^{i+j}u_q\neq 0$ for all $1\leqslant q\leqslant m$,
but $g\cdot u_{m+1}-\xi^{i+j}u_{m+1}=0$. Then the theorem follows.
\end{proof}

\begin{theo}\label{3.7}
Let $\{S_i\}_{0\leqslant i\leqslant t-1}$ be the complete set of non-isomorphic simple
$H(1,\mu)$-modules defined in Theorem \ref{3.4}. Then soc$(S_i\bigotimes
S_j)\cong S_{\overline{i+j}}$ and $S_i\bigotimes S_j$ is
indecomposable. In particular, $S_0\otimes S_i\cong S_i$ and
$S_i\otimes S_0\cong S_i$. Here $0\leqslant i, j\leqslant t-1$.
\end{theo}
\begin{proof}
It is obvious that $S_0\otimes S_i\cong S_i$ and
$S_i\otimes S_0\cong S_i$ for all $0\leqslant i\leqslant t-1$.
Now let $1\leqslant i, j\leqslant t-1$. Let $\{b^{i_1}\cdot v|0\leqslant i_1\leqslant
p-1\}$ and $\{b^{j_1}\cdot w|0\leqslant j_1\leqslant p-1\}$ be the bases of
$S_i$ and $S_j$ as stated in Theorem \ref{3.4}, respectively. Then $\{b^{i_1}\cdot v\otimes
b^{j_1}\cdot w|0\leqslant i_1, j_1\leqslant p-1\}$ is a basis of
$S_i\bigotimes S_j$. By Lemma \ref{3.2}(1), the matrix of the action of $g$ on $S_i\bigotimes S_j$
with respect to the basis
$\{v\otimes w, v\otimes b\cdot w, \cdots, v\otimes b^{p-1}\cdot
w, b\cdot v\otimes w, b\cdot v\otimes b\cdot w, \cdots, b\cdot v\otimes b^{p-1}\cdot
w, \cdots, b^{p-1}\cdot v\otimes w, b^{p-1}\cdot v\otimes b\cdot w,
\cdots, b^{p-1}\cdot v\otimes b^{p-1}\cdot w\}$ has the form
\[
G_0=\left(\begin{array}{cccc}
G_{11}&G_{12}&\cdots&G_{1p}\\
0&G_{22}&\cdots&G_{2p}\\
\cdots&\cdots&\cdots&\cdots\\
0&0&\cdots&G_{pp}
\end{array}\right)
\]
where each $G_{st}$ $(s\leqslant t)$ is a upper triangular $p\times p$-matrix, and $G_{ss}$
has the form
\[
\left(\begin{array}{ccccc}
\\\xi^{i+j}&\alpha_{12}&*&\cdots&*\\
0&\xi^{i+j}&\alpha_{23}&\cdots&*\\
0&0&\xi^{i+j}&\ldots&*\\
\cdots&\cdots&\cdots&\cdots&\cdots\\
0&0&0&\cdots&\xi^{i+j}
\end{array}\right)
\]
with $\alpha_{i_1,i_1+1}\neq 0$.
Hence $\xi^{i+j}$ is the unique eigenvalue of the action of $g$ on $S_i\bigotimes S_j$.
Moreover, $r(\xi^{i+j}I_p-G_{ss})=p-1$. It follows that $r(\xi^{i+j}I_{p^2}-G_0)\geqslant p(p-1)$.
Thus, dim$V_{\xi^{i+j}}\leqslant p$, where $V_{\xi^{i+j}}$ is the eigenspace
of the action of $g$ on $S_i\otimes S_j$.

Obviously, $u_0=v\otimes w\in V_{\xi^{i+j}}$.
For any $1\leqslant l\leqslant p-1$, let $u_{(l)}=b^l\cdot v\otimes w$.
Then $h(u)=u(1)=l>0$. It follows from Lemma \ref{3.2}(1) that $g\cdot
u_{(l)}\neq\xi^{i+j}u_{(l)}$. Then by Theorem \ref{3.6}, there is an element
$u_{(l)}'\in S_i\otimes S_j$ with $h(u_{(l)}')\leqslant l$ and $u_{(l)}'(1)<l$
such that $g\cdot u_l=\xi^{i+j}u_l$, where $u_l=u_{(l)}+u_{(l)}'$.
Obviously, $u_l(1)=l$ and $h(u_l)=l$ for all $0\leqslant l\leqslant p-1$.
It follows that $\{u_0, u_1, \cdots, u_{p-1}\}\subset V_{\xi^{i+j}}$ are linearly independent
over $k$. Thus, $\{u_0, u_1, \cdots, u_{p-1}\}$ is a $k$-basis of $V_{\xi^{i+j}}$.

Let $v_l=g\cdot u_{(l)}-\xi^{i+j}u_{(l)}$. Then it follows from Lemma \ref{3.2} that
$v_l=-l\xi^{i+j}(1-\xi^i)b^{l-1}\cdot v\otimes w
+\sum\limits_{i_1<l-1}\alpha_{i_1}b^{i_1}\cdot v\otimes w$.
Hence $v_l(1)=l-1$ and $v_l(2)=0$. Since $g\cdot u_l-\xi^{i+j}u_l=0$,
$v_l+g\cdot u'_{(l)}-\xi^{i+j}u'_{(l)}=0$. Hence $(g\cdot u'_{(l)}-\xi^{i+j}u'_{(l)})(1)=l-1$
and $(g\cdot u'_{(l)}-\xi^{i+j}u'_{(l)})(2)=0$.
By Lemma \ref{3.2}, we know that $l-1=(g\cdot u'_{(l)}-\xi^{i+j}u'_{(l)})(1)\leqslant u'_{(l)}(1)
<l$, which forces that $u'_{(l)}(1)=l-1$.
Since $u'_{(l)}(1)+u'_{(l)}(2)\leqslant h(u'_{(l)})\leqslant l$, $u'_{(l)}(2)\leqslant 1$.
If $u'_{(l)}(2)=0$, then it follows from Lemma \ref{3.2} that
$l-1=(g\cdot u'_{(l)}-\xi^{i+j}u'_{(l)})(1)< u'_{(l)}(1)=l-1$, a contradiction.
Therefore, $u'_{(l)}(2)=1$, and so $h(u'_{(l)})=l$. Thus we have
$$u'_{(l)}=\alpha b^{l-1}\cdot v\otimes b\cdot w+\beta b^{l-1}\cdot v\otimes w
+\sum\limits_{i_1<l-1}\alpha_{i_1,j_1}b^{i_1}\cdot v\otimes b^{j_1}\cdot w.$$
Again by Lemma \ref{3.2}, one gets
$$g\cdot u'_{(l)}-\xi^{i+j}u'_{(l)}=-\alpha\xi^{i+j}(1-\xi^j) b^{l-1}\cdot v\otimes w
+\sum\limits_{i_1<l-1}\beta_{i_1,j_1}b^{i_1}\cdot v\otimes b^{j_1}\cdot w.$$
Since $v_l+g\cdot u'_{(l)}-\xi^{i+j}u'_{(l)}=0$, $\alpha=-l(1-\xi^i)(1-\xi^j)^{-1}$, and hence
$$u'_{(l)}=-l(1-\xi^i)(1-\xi^j)^{-1}b^{l-1}\cdot v\otimes b\cdot w+\beta b^{l-1}\cdot v\otimes w
+\sum\limits_{i_1<l-1}\alpha_{i_1,j_1}b^{i_1}\cdot v\otimes b^{j_1}\cdot w.$$
Since $ga=ag$, $a\cdot V_{\xi^{i+j}}\subseteq V_{\xi^{i+j}}$. Consider
the action of $a$ on $V_{\xi^{i+j}}$. Then $a\cdot u_0=(1-\xi^{i+j})u_0$.
For $1\leqslant l\leqslant p-1$, let
$u=u_l+\alpha_1u_{l-1}+\ldots+\alpha_lu_0$ be an element in $V_{\xi^{i+j}}$.
If $a\cdot u=\alpha u$ for some $\alpha\in k$, then by comparing their coefficients of
the item $b^l\cdot v\otimes w$, we find that $\alpha=1-\xi^{i+j}$.
It follows that $1-\xi^{i+j}$ is the unique eigenvalue for the action
of $a$ on $V_{\xi^{i+j}}$. Using Lemma \ref{3.2}, one finds that
the coefficient of the item
$b^{l-1}\cdot v\otimes w$ in $a\cdot u-(1-\xi^{i+j})u$ is
$-\frac{l}{2}(1-\xi^i)(1-\xi^{i+j})$. We divide the discussion into the following two
cases.

For case 1: $i+j\neq t$. In this case, $a\cdot u-(1-\xi^{i+j})u\neq0$, and hence
$u$ is not an eigenvector of the action of $a$. It follows that $u_0$
is the unique common eigenvector of the action of $g$ and $a$ up to a non-zero
scale multiple. It follows from Theorem \ref{3.4} that
soc$(S_i\otimes S_j)\cong S_{\overline{i+j}}$.

For case 2: $i+j=t$. In this case, $1$ is the unique eigenvalue of
the action of $g$. It follows from Theorem \ref{3.4} that any simple
submodule of $S_i\otimes S_j$ is isomorphic to $S_0$, and is spanned
by a non-zero vector $v'$ with $g\cdot v'=v'$, $a\cdot v'=0$ and
$b\cdot v'=0$. Now we have $g\cdot u_0=u_0$ and $a\cdot u_0=0$. By
Lemma \ref{2.2}(2), it follows that $g\cdot(b^l\cdot u_0)=b^l\cdot
u_0$ and $a\cdot(b^l\cdot u_0)=0$ for all $1\leqslant l\leqslant
p-1$. Since $\Delta(b)={b\otimes 1}+{g\otimes b}$, one can see that
$(b^l\cdot u_0)(1)=l$, $(b^l\cdot u_0)(2)=0$. It follows that
$\{u_0, b\cdot u_0, \cdots, b^{p-1}\cdot u_0 \}$ are linearly
independent and contained in $V_{\xi^{i+j}}=V_1$. Furthermore,
$b\cdot(b^{p-1}\cdot u_0)=b^p\cdot u_0=0$. Thus, soc$(S_i\otimes
S_j)=k(b^{p-1}\cdot u_0)\cong S_0$.

This completes the proof.
\end{proof}

Now we are going to investigate the indecomposable projective modules over $H(\lambda, \mu)$.

Let $e_i=\frac{1}{t}\sum_{j=0}^{t-1}(\xi^{-ip^s}g^{p^s})^j$. Then
$\{e_0, e_1, \cdots, e_{t-1}\}$ is a set of primitive orthogonal
idempotents in $kG$ since $\xi^{p^s}$ is also a $t$-$th$ primitive
root of unity. Now we have
$(1-\xi^{-ip^s}g^{p^s})e_i=\frac{1}{t}[1-(\xi^{-ip^s}g^{p^s})^t]=0$,
that is, $g^{p^s}e_i=\xi^{ip^s}e_i$. Hence $\{g^{i_1}e_i|0\leqslant
i_1\leqslant p^s-1\}$ is a basis of $kGe_i$ and dim$kGe_i=p^s$.
Under this basis, the matrix of the action of $g$ on $kGe_i$ is
\[
\left(\begin{array}{ccccc}
0&0&\cdots&0&\xi^{ip^s}\\
1&0&\cdots&0&0\\
\cdots&\cdots&\cdots&\cdots&\cdots\\
0&0&\cdots&0&0\\
0&0&\cdots&1&0
\end{array}\right)_{p^s\times p^s}.
\]
The characteristic polynomial of $g$ is
$p(x)=x^{p^s}-\xi^{ip^s}=(x-\xi^i)^{p^s}$. Acting on $kGe_i$, $g$
has a unique eigenvalue $\xi^i$ with multiplicity $p^s$. By Lemma
\ref{2.2}, $g^p\in Z(H(\lambda,\mu))$, the center of
$H(\lambda,\mu)$. Hence $\{e_0, e_1, \cdot, e_{t-1}\}$ is a set of
central orthogonal idempotents of $H(\lambda, \mu)$. It follows that
$H(\lambda,\mu)=\bigoplus_{0\leqslant i\leqslant
t-1}H(\lambda,\mu)e_i$ is a decomposition of the left regular module
$H(\lambda,\mu)$, which is also a composition of $H(\lambda, \mu)$
as two-sided ideals. Thus, the action of $g$ on $H(\lambda,\mu)e_i$
has the unique eigenvalue $\xi^i$ (with multiplicity of $p^{s+2}$).
So $g$ has the unique eigenvalue $\xi^i$ when it acts on every
principal projective module occurring in $H(\lambda,\mu)e_i$.

Note that dim$H(\lambda, \mu)$=dim$(\mathcal{B}(V)\# kG)=p^2n=p^{s+2}t$
and
$$H(\lambda, \mu)e_i={\rm span}\{g^{i_1}a^{i_2}b^{i_3}e_i|
0\leqslant i_1\leqslant p^s-1, 0\leqslant i_2, i_3\leqslant
p-1\}.$$
Hence dim$H(\lambda, \mu)e_i=p^{s+2}$
and $\{g^{i_1}a^{i_2}b^{i_3}e_i|
0\leqslant i_1\leqslant p^s-1, 0\leqslant i_2, i_3\leqslant p-1\}$ is a basis of $H(\lambda,\mu)e_i$.

Now we can prove the main results of this section.

\begin{theo}\label{3.8}
Let $\{T_0, T_1, \cdots, T_{t-1}\}$ be the complete set of
non-isomorphic simple $H(0,\mu)$-modules given in Theorem \ref{3.3}.
Let $P(T_i)$ denote the projective cover of $T_i$. Then $P(T_i)\cong
H(0,\mu)e_i$, where $0\leqslant i\leqslant t-1$.
\end{theo}
\begin{proof}
Since $\xi^i$ is an eigenvalue of the action of $g$ on
$T_i\cong P(T_i)/{\rm rad}(P(T_i))$, $\xi^i$ is the unique eigenvalue of
the action of $g$ on $P(T_i)$. It follows that $P(T_i$) must be the unique summand of $H(0,\mu)e_i$
up to isomorphism of $H(0, \mu)$-modules.
Since dim$T_i=1$, the left regular module $H(0, \mu)$ has the decomposition
$H(0, \mu)\cong\bigoplus_{0\leqslant i\leqslant t-1}P(T_i)$, which forces that $P(T_i)\cong H(0,\mu)e_i$.
\end{proof}

Now we are going to consider the case of $\lambda=1$. Let us first
show the following lemma for the case of $\mu=0$.

\begin{lem}\label{3.9-1}
In the Hopf algebra $H(1, 0)$, we have
$$b^mab^{p-1}=\frac{m!}{2^m}a^{m+1}b^{p-1},\ \ m\geqslant 0.$$
\end{lem}

\begin{proof}
We prove the equation $b^mab^{p-1}=\frac{m!}{2^m}a^{m+1}b^{p-1}$ by
induction on $m$. If $m=0$, it is obvious. Now let $m\geqslant 0$
and assume $b^mab^{p-1}=\frac{m!}{2^m}a^{m+1}b^{p-1}$. Since
$b^p=0$, by Lemma \ref{2.2}(1) we have
$$\begin{array}{rcl}
b^{m+1}ab^{p-1}&=&\frac{m!}{2^m}ba^{m+1}b^{p-1}\\
&=&\frac{m!}{2^m}(a^{m+1}b+\frac{m+1}{2}a^{m+2})b^{p-1}\\
&=&\frac{m!}{2^m}(a^{m+1}b^p+\frac{m+1}{2}a^{m+2}b^{p-1})\\
&=&\frac{(m+1)!}{2^{m+1}}a^{m+2}b^{p-1}.\\
\end{array}$$
This completes the proof.
\end{proof}

\begin{theo}\label{3.9}
Let $\{S_0, S_1, \cdots, S_{t-1}\}$ be the complete set of
non-isomorphic simple $H(1,\mu)$-modules described as Theorem \ref{3.4}.
Let $P(S_i)$ denote the projective cover of $S_i$. Then

$(1)$ $P(S_0)\cong H(1,\mu)e_0$ and ${\rm dim}P(S_0)=p^{s+2}$.

$(2)$ Let $1\leqslant i\leqslant t-1$. Then ${\rm dim}P(S_i)=p^{s+1}$.
Moreover, if $\mu=0$, then
$P(S_i)\cong H(1,0)b^{p-1}e_i$ and
$\{g^{i_1}a^{i_2}b^{p-1}e_i|0\leqslant i_1\leqslant p^s-1, 0\leqslant
i_2\leqslant p-1\}$ is a basis of $H(1, 0)b^{p-1}e_i$.
If $\mu\neq0$ and $s=1$, then $P(S_i)\cong H(1,\mu)b_0^{p-1}e_i$,
and $H(1,\mu)b_0^{p-1}e_i$ has a basis $\{g^{i_1}a^{i_2}b^{p-1}e_i|
0\leqslant i_1\leqslant p^s-1, 0\leqslant i_2\leqslant
p-1\}$, where $b_0=b+\alpha_0$ and $\alpha_0=\mu^{\frac{1}{p}}(\xi^i-1)$.
\end{theo}

\begin{proof}
(1) Since $\xi^i$ is an eigenvalue of the action of $g$ on
$S_i=P(S_i)/{\rm rad}(P(S_i))$, $\xi^i$ is the unique eigenvalue of the action of $g$ on
$P(S_i)$. It follows that $P(S_i$) must be the unique summand of $H(1,\mu)e_i$
up to the isomorphism of $H(1, \mu)$-modules. By
Wedderburn-Artin Theorem, the left regular module $H(1, \mu)$ has the decomposition
$H(1, \mu)\cong\bigoplus_{0\leqslant i\leqslant t-1}P(S_i)^{{\rm dim}S_i}$,
where $P(S_i)^m$ denotes the direct sum of $m$ copies of $P(S_i)$.
It follows that $H(1, \mu)e_i\cong P(S_i)^{{\rm dim}S_i}$ as left $H(1, \mu)$-modules.
Since dim$S_0=1$, one gets that $P(S_0)\cong H(1,\mu)e_0$ and dim$P(S_0)=p^{s+2}$.

(2) Let $1\leqslant i\leqslant t-1$. Since dim$S_i=p$ and dim$H(1, \mu)e_i=p^{s+2}$,
$H(1,\mu)e_i\cong P(S_i)^p$, the direct sum of $p$ copies of $P(S_i)$. Hence dim$P(S_i)=p^{s+1}$.

Assume $\mu=0$. Then by Lemma \ref{3.9-1} we have
$b^{p-1}ab^{p-1}=\frac{(p-1)!}{2^{p-1}}a^pb^{p-1}$. Let
$\widetilde{e_i}=a^{p^s-p+1}b^{p-1}e_i$. Since $a^p=1-g^p$ and $g^p\in Z(H(1, 0))$, we have
$a^p\in Z(H(1, 0))$. Therefore, we have
\begin{equation*}
\begin{split}
\widetilde{e_i}^2&=a^{p^s-p+1}b^{p-1}a^{p^s-p+1}b^{p-1}e_i\\
&=a^{2(p^s-p)+1}b^{p-1}ab^{p-1}e_i\\
&=\frac{(p-1)!}{2^{p-1}}a^{2(p^s-p)+1}a^pb^{p-1}e_i\\
&=\frac{(p-1)!}{2^{p-1}}a^{p^s-p+1}a^{p^s}b^{p-1}e_i\\
&=\frac{(p-1)!}{2^{p-1}}a^{p^s-p+1}b^{p-1}(1-g^{p^s})e_i\\
&=\frac{(p-1)!}{2^{p-1}}(1-\xi^{ip^s})a^{p^s-p+1}b^{p-1}e_i\\
&=\frac{(p-1)!}{2^{p-1}}(1-\xi^{ip^s})\widetilde{e_i}.
\end{split}
\end{equation*}
Then $\widetilde{e_i}^2=\alpha\widetilde{e_i}$ with
$\alpha=\frac{(p-1)!}{2^{p-1}}(1-\xi^{ip^s})\neq0$ in $k$. Let
$\widehat{e_i}=\alpha^{-1}\widetilde{e_i}$. Then
$\widehat{e_i}^2=\widehat{e_i}$. Hence $H(1, 0)\widehat{e_i}$ is a
summand of $H(1, 0)e_i$ as a left $H(1, 0)$-module. It follows that
$H(1, 0)\widehat{e_i}\cong P(S_i)^m$ for some $1\leqslant
m\leqslant{\rm dim}S_i$. Obviously, $H(1, 0)\widehat{e_i}\subseteq
H(1, 0)b^{p-1}e_i$. Since $a^p=1-g^p$ and $b^p=0$, it follows from
Lemma \ref{2.2}(1) that $H(1, 0)b^{p-1}e_i={\rm
span}\{g^{i_1}a^{i_2}b^{p-1}e_i|0\leqslant i_1\leqslant p^s-1,
0\leqslant i_2\leqslant p-1\}$. Hence $p^{s+1}={\rm dim}P(S_i)
\leqslant{\rm dim}(H(1,0)\widehat{e_i})\leqslant{\rm
dim}(H(1,0)b^{p-1}e_i) \leqslant p^{s+1}$. This implies that ${\rm
dim}(H(1,0)\widehat{e_i})={\rm dim} (H(1, 0)b^{p-1}e_i)=p^{s+1}$.
Hence $P(S_i)\cong H(1,0)\widehat{e_i}=H(1, 0)b^{p-1}e_i$, and
consequently $H(1,0)b^{p-1}e_i$ has a basis
$\{g^{i_1}a^{i_2}b^{p-1}e_i|0\leqslant i_1\leqslant p^s-1,
0\leqslant i_2\leqslant p-1\}$.

Now assume $\mu\neq0$ and $s=1$. Let $\alpha_0=\mu^{\frac{1}{p}}(\xi^i-1)\in k\subseteq H(1, \mu)$
and $b_0=b+\alpha_0\in H(1,\mu)$. Then
$b_0^p=\mu(\xi^{ip}-g^p)=\mu(\xi^i-g)^p$, and so $b_0^pe_i=0$.
Since $g^p\in Z(H(1,\mu))$, $b_0^p\in
Z(H(1,\mu))$. An argument similar to Lemma \ref{3.9-1} shows that
$b_0^mab_0^{p-1}e_i=\frac{m!}{2^m}a^{m+1}b_0^{p-1}e_i$ for all
$m\geqslant 0$. Let $e'_i=\frac{2^{p-1}}{(p-1)!}(1-\xi^{ip})^{-1}ab_0^{p-1}e_i$.
Then it follows from an argument similar to the case of $\mu=0$ that
$(e'_i)^2=e'_i$, $P(S_i)\cong H(1,\mu)e'_i=H(1,\mu)b_0^{p-1}e_i$ and
$\{g^{i_1}a^{i_2}b_0^{p-1}e_i|0\leqslant i_1\leqslant p^s-1, 0\leqslant i_2\leqslant
 p-1\}$ is a basis of $H(1, \mu)b_0^{p-1}e_i$.
\end{proof}
\begin{rem}If $p=3,5,7,11$, we find that $b_1^p=[b+\mu^{\frac{1}{p}}(g-1)]^p=0$.
Then the argument in the proof of Theorem \ref{3.9} can be applied to $H(1,\mu)$
with $\mu\neq0$ and $s\geqslant 1$. In this case,
we have that $P(S_i)\cong H(1,\mu)b_1^{p-1}e_i$ and $\{g^{i_1}a^{i_2}b_1^{p-1}e_i|
0\leqslant i_1\leqslant p^s-1, 0\leqslant i_2\leqslant p-1\}$
is a basis of $H(1, \mu)b_1^{p-1}e_i$, where $1\leqslant i\leqslant t-1$.
\end{rem}

\begin{col}\label{3.11} If $t>1$, then $\{e_0, e_1, \cdots, e_{t-1}\}$ is a set of central orthogonal
primitive idempotents of $H(\lambda,\mu)$.
\end{col}

\begin{col}\label{3.11} If $t>1$, then each block $H(\lambda,\mu)e_i$ of $H(\lambda,\mu)$ is a
symmetric algebra. Moreover,
$H(\lambda,\mu)e_0$ is a local symmetric algebra.
\begin{proof}
It follows from Lemma \ref{2.3} and \cite[Lemma I.3.3]{Erd90}
\end{proof}
\end{col}

\section{Representation types of $\mathcal{B}(V)\# kG$ and $H(\lambda,\mu)$}\label{4}

In this section, we will consider the representation types of
$\mathcal{B}(V)\# kG$ and $H(\lambda,\mu)$. Let us first consider
the simple modules and their projective covers over $\mathcal{B}(V)\# kG$.
When $p>2$, $\mathcal{B}(V)\# kG=H(0, 0)$ as noted in the last section.
In this case, the simple modules and their projective covers over $\mathcal{B}(V)\# kG$
have been described in the last section, see Theorems \ref{3.3} and \ref{3.8}.

Now let us assume $p=2$ and $n=2^st$ with $2\nmid t$ and $s\geqslant 1$. Let $\xi$
be a $t$-$th$ primitive root of unity in $k$. We denote by $H$ the
Hopf algebra $\mathcal{B}(V)\# kG$ defined in Section \ref{2}.

Since $H$ is a finite dimensional graded Hopf algebra $H=\bigoplus_{m\geqslant 0}H_m$
with $H_0=kG$ and $a, b\in H_1$, a left $H$-module $M$ is a simple $H$-module
if and only if $M$ is a simple $kG$-module and $a\cdot M=b\cdot M=0$.
Hence we have the following proposition.

\begin{prop}\label{4.1}
Up to isomorphism, there are $t$ simple left $H$-modules $S_i$,
which are all 1-dimensional and defined by
$$g\cdot x=\xi^ix,\ a\cdot x=b\cdot x=0,\ x\in S_i,$$
where $0\leqslant i\leqslant t-1$. In particular,
if $t=1$, then $H$ is a local algebra.
\end{prop}

Let $e_i=\frac{1}{t}\sum_{j=0}^{t-1}(\xi^{-i2^s}g^{2^s})^j$. Then
$\{e_0, e_1, \cdots, e_{t-1}\}$ is a set of primitive orthogonal
idempotents in $kG$ and $g^{2^s}e_i=\xi^{i2^s}e_i$.
$\{g^{i_1}e_i|0\leqslant i_1\leqslant 2^s-1\}$ is a basis of $kGe_i$
and dim$kGe_i=2^s$. By Lemma \ref{2.1}, $g^2\in Z(H)$, the center of
$H$. Hence $\{e_0, e_1, \cdots, e_{t-1}\}$ is a set of central
orthogonal idempotents of $H$. It follows that
$H=\bigoplus_{0\leqslant i\leqslant t-1}He_i$ is a decomposition of
the left regular module $H$, which is also a composition of $H$ as
two-sided ideals. By a discussion similar to that for $H(\lambda,\mu)$ in
Section \ref{3}, we have the following result from Lemma \ref{2.1} and
\cite[Lemma I.3.3]{Erd90}.

\begin{theo}\label{4.2}
Let $\{S_0, S_1, \cdots, S_{t-1}\}$ be the complete set of
non-isomorphic simple $H$-modules given in Proposition \ref{4.1}.
Let $P(S_i)$ denote the projective cover of $S_i$. Then

$(1)$ $P(S_i)\cong He_i$, where $0\leqslant i\leqslant t-1$.

$(2)$ $H$ has $t$ blocks $He_i$. Moreover, each block $He_i$ is a local
symmetric algebra.
\end{theo}

\begin{lem}\label{4.3}
Let $0\leqslant i\leqslant t-1$. Let $M$ be an indecomposable module of dimension $2$
over the block $He_i$. Then $M$
has one of the following structures:

$(1)$ There is a $k$-basis $\{v_1, v_2\}$ in $M$ such that
$g\cdot v_1=\xi^i\cdot v_1$, $g\cdot v_2=\xi^i\cdot v_2$, $a\cdot v_1=$
$a\cdot v_2=0$, $b\cdot v_1=0$ and $b\cdot v_2=v_1$.

$(2)$ There is a $k$-basis $\{v_1,v_2\}$ in $M$ such that
$g\cdot v_1=\xi^iv_1$, $g\cdot v_2=\xi^iv_2+v_1$,
$a\cdot v_1=a\cdot v_2=0$, $b\cdot v_1=0$ and
$b\cdot v_2=\gamma v_1$ for some $\gamma\in k$.
\end{lem}
\begin{proof}
Let $M$ be a left $He_i$-module of dimension 2. Then $M$ is
a $kGe_i$-module. Since $g^{2^s}e_i=\xi^{i2^s}e_i$, there is a basis $\{v_1, v_2\}$ of $M$ such that
the corresponding matrix $G_1$ of the action of $g$ on $M$ is one of the followings:
\[\begin{matrix}
\begin{pmatrix} \xi^i & 0\\0 & \xi^i \end{pmatrix},&
\begin{pmatrix} \xi^i & 1\\0 & \xi^i \end{pmatrix}.&
\end{matrix}\]
Let $A$ and $B$ denote the matrices of the actions of $a$ and $b$
with respect to the basis $\{v_1, v_2\}$ of $M$, respectively.

Assume $G_1=\begin{pmatrix} \xi^i & 1\\0 & \xi^i
\end{pmatrix}$. Since $ga=ag$, $AG_1=G_1A$. Hence $A=\begin{pmatrix}\alpha_1 & \alpha_2\\0 & \alpha_1
\end{pmatrix}$ for some $\alpha_1, \alpha_2\in k$. Since $a^2=0$, $A$ is
a nilpotent matrix, and so $\alpha_1=0$. From $bg=ga+gb$, one knows that $BG_1=G_1B+G_1A$.
Then it follows that $B=\begin{pmatrix} \beta+\xi^i\alpha_2 & \gamma\\0 & \beta
\end{pmatrix}$ for some $\beta, \gamma\in k$. Since $b^4=0$, $B$ is a nilpotent matrix.
Hence $\beta+\xi^i\alpha_2=\beta=0$, and so $\alpha_2=0$. Thus,
$A=0$ and $B=\begin{pmatrix} 0 & \gamma\\0 & 0
\end{pmatrix}$. In this case, $M$ has the structure described in (2).

Assume $G_1=\begin{pmatrix} \xi^i & 0\\0 & \xi^i
\end{pmatrix}$. Then $G_1B=BG_1$. Since $BG_1=GB+G_1A$, $G_1A=0$, and so $A=0$.
In this case, under any basis of $M$, the matrix of the action of
$g$ is always $G_1$ and $A$ is always 0. If $b\cdot M=0$, then
$M\cong S_i\oplus S_i$, a semisimple module. Hence $b\cdot M\neq 0$.
So we may choose a basis $\{v_1, v_2\}$ of $M$ such that
$B=\begin{pmatrix}0 & 1\\0 & 0\end{pmatrix}$ since $b$ is a
nilpotent element of $H$. Thus, $M$ has the structure described in
(1).

This completes the proof.
\end{proof}

Let $0\leqslant i\leqslant t-1$. For $\gamma\in k$, let $M(\gamma)$ denote the 2-dimensional
module over the block $He_i$ described as in Lemma \ref{4.3}(2).

\begin{lem}\label{4.4}
Let $0\leqslant i\leqslant t-1$ and $\gamma_1, \gamma_2\in k$. Then $M(\gamma_1)\cong M(\gamma_2)$
if and only if $\gamma_1=\gamma_2$.
\end{lem}

\begin{proof}
Let $G_1=\begin{pmatrix} \xi^i & 1\\0 & \xi^i \end{pmatrix}$,
$B_1=\begin{pmatrix} 0 & \gamma_1\\0 & 0\end{pmatrix}$ and
$B_2=\begin{pmatrix} 0 & \gamma_2\\0 & 0\end{pmatrix}$.
If $M(\gamma_1)\cong M(\gamma_2)$, there exists an invertible
matrix $F\in M_2(k)$ such that $G_1F=FG_1$ and $B_1F=FB_2$.
Then one can get that $\gamma_1=\gamma_2$.
\end{proof}

\begin{rem}\label{4.5}
Let $0\leqslant i\leqslant t-1$ and $\beta, \gamma\in k$.
Then there is an algebra map $f: H\rightarrow M_2(k)$ defined by
$$f(g)=\begin{pmatrix} \xi^i & \beta\\0 & \xi^i\end{pmatrix},\
f(a)=\begin{pmatrix} 0 & 0\\0 & 0\end{pmatrix},\
f(b)=\begin{pmatrix} 0 & \gamma\\0 & 0\end{pmatrix}.$$
Let $M(\beta, \gamma)$ denote the corresponding $H$-module.
Obviously, $M(\beta, \gamma)$ is a module over the block $He_i$.
One can easily check that $M(\beta, \gamma)\cong M(\beta', \gamma')$
if and only if $(\beta, \gamma)=\alpha(\beta', \gamma')$ in $k\times k$
for some $0\neq\alpha\in k$. If $\beta=\gamma=0$,
then $M(\beta, \gamma)\cong S_i\oplus S_i$. Otherwise, $M(\beta, \gamma)$
is indecomposable.

Let $\{v_1, v_2\}$ be the basis of $M(\beta, \gamma)$ such that
the corresponding matrix representation are given as above. Fix a
non-zero element $v\in S_i$. Then there is an exact sequence
$$0\rightarrow S_i\xrightarrow{\theta} M(\beta, \gamma)\xrightarrow{\eta}S_i\rightarrow 0$$
given by $\theta(v)=v_1$, $\eta(v_1)=0$ and $\eta(v_2)=v$.
Denote by $E(\beta, \gamma)$ the extension of $S_i$ by $S_i$.
Then a straightforward verification shows that two extensions $E(\beta, \gamma)$
and $E(\beta', \gamma')$ are equivalent if and only if $(\beta, \gamma)=(\beta', \gamma')$.
Thus, we have the following corollary.
\end{rem}

\begin{col}\label{4.6}
Let $0\leqslant i, j\leqslant t-1$. Then
$${\rm dim(Ext}(S_i,S_j))=
\begin{cases} 2, &{\rm if }\ i=j \\0, &{\rm if }\ i\neq j
\end{cases}$$
\end{col}

Now we will consider the representation type of $H$. Since $H$ has $t$
blocks $He_i$, we only need to consider the representation type of
each block $He_i$. Let
$$I=\{1, a, b, ab, ba, b^2, aba, ab^2, bab,b^3, abab, ab^3, bab^2, abab^2, bab^3, abab^3\}.$$
Then by \cite[Theorem 3.1 and Corollary 3.4]{Cib09},
$H$ is a $2^{s+4}t$-dimensional graded Hopf algebra with a basis
$\{g^jx|0\leqslant j\leqslant 2^st -1, x\in I\}$.
Since $g^{2^s}e_i=\xi^{i2^s}e_i$, by a discussion similar to that for $H(\lambda,\mu)$ in
Section \ref{3}, one gets that each block $He_i$ is $2^{s+4}$-dimensinal with a basis
$\{g^jxe_i|0\leqslant j\leqslant 2^s-1, x\in I\}$, where $0\leqslant i\leqslant t-1$.

Note that ${\rm deg}(a)={\rm deg}(b)=1$ in the graded Hopf algebra $H$.

\begin{theo}\label{4.7}
Let $0\leqslant i\leqslant t-1$.
Then the block $He_i$ is of wild representation type.
\end{theo}

\begin{proof}
Let $0\leqslant i\leqslant t-1$. Then $\{(g-\xi^i)^jxe_i|0\leqslant j\leqslant 2^s-1, x\in I\}$ is
also a basis of $He_i$.
Let $J$ denote the Jacobson radical of $He_i$. Since $S_i$ is the unique simple module
over the block $He_i$, it follows from Proposition \ref{4.1} that $J$ has a basis
$\{(g-\xi^i)^jxe_i|0\leqslant j\leqslant 2^s-1, x\in I, j+{\rm deg}(x)\geqslant 1\}$.
Since $g^{-1}bg=a+b$, we
have $b(g-\xi^i)^m=(g-\xi^i)^mb+m(g-\xi^i)^ma+m\xi^i(g-\xi^i)^{m-1}a$ for all $m\geqslant 1$ by induction
on $m$. By these relations and the other relations of $H$, it is easy to
check that $N={\rm span}\{(g-\xi^i)^jxe_i, ae_i|0\leqslant j\leqslant
2^s-1, j+{\rm deg}(x)\geqslant 2\}$ is a left ideal of $He_i$ and $N\subseteq J^2$.
Observe that ${\rm dim}(J/N)=2$. By \cite[Proposition III.1.14]{Aus95} and Corollary \ref{4.6},
we have ${\rm dim}(J/J^2)={\rm dim(Ext}(S_i,S_i))=2$.
It follows that $J^2=N$. Let $M={\rm span}\{
(g-\xi^i)^jxe_i, (g-\xi^i)ae_i, abe_i, bae_i
|0\leqslant j\leqslant 2^s-1, j+{\rm deg}(x)\geqslant 3\}$.
Then it is easy to check that $M$ is a left ideal of $He_i$ and
$M\subseteq J^3$. Moreover, one can check that $J^2/M$ is a semisimple $He_i$-module,
and so $J^3\subseteq M$. Thus $J^3=M$. Obviously,
$J^2/M={\rm span}\{\overline{ae_0}, \overline{(g-\xi^i)^2e_i}, \overline{(g-\xi^i)be_i}, \overline{b^2e_i}\}$,
where $\overline{y}=y+M$ in $J^2/M$ for any $y\in J^2$.
Note that $(g-\xi^i)^2e_i=0$ when $s=1$. Hence
$3\leqslant{\rm dim}(J^2/M)\leqslant 4$. Since $He_i$ is a local
symmetric algebra by Theorem \ref{4.2} and ${\rm dim}(J^2/J^3)\geqslant3$,
it follows from \cite[Lemma III.4]{Erd90} that $He_i$ is of wild representation type.
\end{proof}

\begin{col}\label{4.8}
Assume $p=2$. Then $\mathcal{B}(V)\# kG$ is of wild representation type.
\end{col}

In the rest of this section, assume $p>2$ and $\lambda, \mu\in k$.
We will consider the representation type
of $H(\lambda, \mu)$. Let $\{e_0, e_1, \cdots, e_{t-1}\}$ be the set
of central orthogonal primitive idempotents of $H(\lambda, \mu)$
described as in the last section. Then $H(\lambda, \mu)$ has $t$
blocks $H(\lambda, \mu)e_i$. Hence we only need to consider the
representation type of each block $H(\lambda, \mu)e_i$.
We first consider the case of $\lambda=0$. From Theorems \ref{3.3}
and \ref{3.8}, one knows that $H(0, \mu)$ is a basic algebra and that
$T_i$ is the unique simple module over the block $H(0, \mu)e_i$,
where $0\leqslant i\leqslant t-1$. Moreover, each block $H(0, \mu)e_i$
is a local symmetric algebra by Lemma \ref{2.3} and \cite[Lemma I.3.3]{Erd90}.

\begin{lem}\label{4.9}
We have ${\rm dim(Ext}(T_i, T_i))=2$ over each block $H(0,\mu)e_i$,
where $0\leqslant i\leqslant t-1$.
\end{lem}

\begin{proof}
Let $0\leqslant i\leqslant t-1$. Then it follows from Theorems \ref{3.3}
and \ref{3.8} that there is only one simple module $T_i$
over the block $H(0,\mu)e_i$. Let $\beta, \gamma\in k$. Then there is an algebra map
$f: H(0,\mu)\rightarrow M_2(k)$ defined by
$$f_{\gamma}(g)=\begin{pmatrix} \xi^i & \beta\\0 & \xi^i\end{pmatrix},\
f_{\gamma}(a)=\begin{pmatrix} 0 &0\\0 & 0\end{pmatrix},\
f_{\gamma}(b)=\begin{pmatrix}\mu^{\frac{1}{p}}(1-\xi^i) & \gamma\\
0 &\mu^{\frac{1}{p}}(1-\xi^i)\end{pmatrix}.$$
Let $N(\beta, \gamma)$ be the corresponding $H(0, \mu)$-module.
Obviously, $N(\beta, \gamma)$ is a module over the block $H(0, \mu)e_i$.
An argument similar to $H$ shows that any 2-dimensional module over the block $H(0, \mu)e_i$
is isomorphic to some $N(\beta, \gamma)$ and that $N(\beta, \gamma)\cong N(\beta', \gamma')$
if and only if $(\beta, \gamma)=\alpha(\beta', \gamma')$ for some $0\neq \alpha\in k$.
It
follows that ${\rm dim(Ext}(T_i,T_i))=2$ from an argument similar to the
case $p=2$.
\end{proof}

\begin{theo}\label{4.10}
Each block $H(0,\mu)e_i$ is of wild representation type,
where $0\leqslant i\leqslant t-1$.
\end{theo}
\begin{proof}
Let $0\leqslant i\leqslant t-1$.
Since $\{g^{i_1}a^{j_1}b^{k_1}e_i|0\leqslant i_1\leqslant p^s-1, 0\leqslant j_1, k_1\leqslant p-1\}$
is a basis of $H(0, \mu)e_i$, $\{(g-\xi^i)^{i_1}a^{j_1}(b-\mu^{\frac{1}{p}}(1-\xi^i))^{k_1}e_i
|0\leqslant i_1\leqslant p^s-1, 0\leqslant j_1, k_1\leqslant p-1\}$
is also a basis of $H(0, \mu)e_i$.
Let $J$ denote the Jacobson radical of $H(0, \mu)e_i$.
Then it follows from Theorem \ref{3.3} that the set
$$\left\{(g-\xi^i)^{i_1}a^{j_1}(b-\mu^{\frac{1}{p}}(1-\xi^i))^{k_1}e_i
\left|\begin{array}{l}
0\leqslant i_1\leqslant p^s-1,\\
0\leqslant j_1, k_1\leqslant p-1,\\
1\leqslant i_1+j_1+k_1\\
\end{array}\right.\right\}$$
is a basis of $J$.
From $g^{-1}bg=a+b$ and $ba=ab+\frac{1}{2}a^2$, one can easily check that
$$\begin{array}{ll}
&(b-\mu^{\frac{1}{p}}(1-\xi^i))(g-\xi^i)^m\\
=&(g-\xi^i)^m(b-\mu^{\frac{1}{p}}(1-\xi^i))+m(g-\xi^i)^ma+m\xi^i(g-\xi^i)^{m-1}a\\
\end{array}$$
for all $m\geqslant 1$ and
$$(b-\mu^{\frac{1}{p}}(1-\xi^i))a=a(b-\mu^{\frac{1}{p}}(1-\xi^i))+\frac{1}{2}a^2.$$
Put
$$N={\rm span}\left\{(g-\xi^i)^{i_1}a^{j_1}(b-\mu^{\frac{1}{p}}(1-\xi^i))^{k_1}e_i, ae_i
\left|\begin{array}{l}0\leqslant i_1\leqslant p^s-1,\\
0\leqslant j_1, k_1\leqslant p-1,\\
2\leqslant i_1+j_1+k_1\\
\end{array}\right.\right\}.$$
Then from the first one of the above two equalities, one can see that $N\subseteq J^2$.
Obviously, ${\rm dim}(J/N)=2$. By
\cite[Proposition III.1.14]{Aus95} and Lemma \ref{4.9}, we have
${\rm dim}(J/J^2)={\rm dim(Ext}(T_i, T_i))=2$.
It follows that $J^2=N$. Now put
$$M={\rm span}\left\{\begin{array}{l}
(g-\xi^i)^{i_1}a^{j_1}(b-\mu^{\frac{1}{p}}(1-\xi^i))^{k_1}e_i,\\
(g-\xi^i)ae_i, a^2e_i, a(b-\mu^{\frac{1}{p}}(1-\xi^i))e_i\\
\end{array}
\left|\begin{array}{l}0\leqslant i_1\leqslant p^s-1,\\
0\leqslant j_1, k_1\leqslant p-1,\\
3\leqslant i_1+j_1+k_1\\
\end{array}\right.\right\}.$$
Since $J^2=N$, $M\subseteq J^3$. Now from the two equalities given above and
$ga=ag$, one can check that $M$ is a left ideal of $H(0,\mu)e_i$ and
$J^2/M$ is a semisimple module over $H(0, \mu)e_i$. Hence
$J^3\subseteq M$ and so $J^3=M$. Obviously,
$$J^2/M={\rm span}\left\{\overline{(g-\xi^i)^2e_i}, \overline{ae_i},
\overline{(g-\xi^i)(b-\mu^{\frac{1}{p}}(1-\xi^i))e_i}, 
\overline{(b-\mu^{\frac{1}{p}}(1-\xi^i))^2e_i}\right\}$$
is 4-dimensional, where $\overline{x}=x+M$ in $J^2/M$ for any $x\in J^2$.
Hence ${\rm dim}(J^2/J^3)=4$.
Since $H(0,\mu)e_i$ is a local symmetric algebra, it follows from
\cite[Lemma III.4]{Erd90} that $H(0, \mu)e_i$ is of wild
representation type.
\end{proof}

Now we consider the case of $\lambda\neq 0$. We only consider the representation type
of the block $H(1, \mu)e_0$. From Theorems \ref{3.4} and \ref{3.9},
the trivial module $S_0$ is the unique simple module over the block $H(1, \mu)e_0$,
and $H(1, \mu)e_0$ is a basic and local algebra.
Furthermore, $H(1, \mu)e_0$ is a symmetric algebra by Lemma \ref{2.3} and \cite[Lemma I.3.3]{Erd90}.
Then by setting $i=0$ in the proofs of Lemma \ref{4.9} and Theorem \ref{4.10},
one can get the following Lemma \ref{4.11} and Theorem \ref{4.12}

\begin{lem}\label{4.11}
We have ${\rm dim(Ext}(S_0,S_0))=2$ over the block $H(1,\mu)e_0$.
\end{lem}

\begin{theo}\label{4.12}
The block $H(1, \mu)e_0$ is of wild representation type.
\end{theo}

For the case of $t>1$, we don't know whether $H(1, \mu)e_i$ is of tame or wild
representation type, where $1\leqslant i\leqslant t-1$.

Summarizing the above discussion, we have the following result.

\begin{theo}\label{4.13}
Assume $p>2$. Then $H(\lambda,\mu)$ is of wild representation type
for any $\lambda, \mu\in k$. In particular,
$\mathcal{B}(V)\# kG$ is of wild representation type.
\end{theo}

\section*{\bf Acknowledgment}
This work is supported by NSF of China, No. 10771183, and supported by Doctorate foundation,
No. 200811170001, Ministry of Education of China.

\end{document}